\documentclass[10pt,a4paper]{article}
\usepackage[latin1]{inputenc}
\usepackage{amsmath}
\usepackage{amsfonts}
\usepackage{amssymb}
\usepackage{graphicx}
\usepackage{amsthm}
\usepackage{enumerate}
\usepackage{setspace}
\usepackage{booktabs}
\usepackage{tabularx}
\usepackage{lineno} 
\usepackage{color} 
\usepackage{enumerate}
\usepackage{xcolor}
\usepackage{subfig}
\usepackage{caption}

\usepackage{tikz}
\usetikzlibrary{arrows,shapes,decorations,automata,backgrounds,petri}
\usetikzlibrary{decorations.pathmorphing}

\usepackage[top=1.45in, bottom=1.45in, left=1.15in, right=1.15in]{geometry}
\newtheorem{theorem}{Theorem}[section]

\newtheorem{lemma}{Lemma}[section]

\newtheorem{proposition}{Proposition}[section]
\newtheorem{conjecture}{Conjecture}[section]
\newtheorem{question}{Question}[section]

\newtheorem{claim}{Claim}[]

\newenvironment{proofclaim}[1][]%
    {\noindent \emph{Proof.} {}{#1}{}}{\hfill
    $\Diamond$\vspace{1em}}

\newcommand{\cb}[2]{\lceil \frac{#1}{#2} \rceil}
\newcommand{\fb}[2]{\lfloor \frac{#1}{#2} \rfloor}

\title{Gallai's path decomposition conjecture\\for graphs of small maximum degree}
\author{Marthe Bonamy\thanks{CNRS, LaBRI, Universit\'e de Bordeaux, France} \ and Thomas J. Perrett\thanks{Technical University of Denmark, Denmark}}
\date{September 2016}

\begin{document}

\maketitle

\begin{abstract}
    Gallai's path decomposition conjecture states that the edges of any connected graph on $n$ vertices can be decomposed into at most $\frac{n+1}2$ paths. We confirm that conjecture for all graphs with maximum degree at most five.
\end{abstract}

\section{Introduction}
A \emph{decomposition} $\mathcal{D}$ of a graph $G$ is a collection of subgraphs of $G$ such that each edge belongs to precisely one graph in $\mathcal{D}$. A \emph{path decomposition} is a decomposition $\mathcal{D}$ such that every subgraph in $\mathcal{D}$ is a path. If $G$ has a path decomposition $\mathcal{D}$ such that $|\mathcal{D}|=k$, then we say that $G$ can be \emph{decomposed} into $k$ paths. In answer to a question of Erd\H{o}s, Gallai conjectured the following, see~\cite{Lovasz}.

\begin{conjecture}\label{conj:Gallai}\emph{\cite{Lovasz}}
Every connected graph on $n$ vertices can be decomposed into $\cb{n}{2}$ paths.
\end{conjecture}

Gallai's conjecture is easily seen to be sharp: If $G$ is a graph in which every vertex has odd degree, then in any path decomposition of $G$ each vertex must be the endpoint of some path, and so at least $\cb{n}{2}$ paths are required. Lov\'{a}sz~\cite{Lovasz} proved that every graph on $n$ vertices has a decomposition $\mathcal{D}$ consisting of paths and cycles, and such that $|\mathcal{D}|=\fb{n}{2}$. By an argument similar to the above, it follows that in a graph with at most one vertex of even degree, such a decomposition must be a path decomposition. Thus, Gallai's conjecture holds for all graphs with at most one vertex of even degree.

Let $G_E$ denote the subgraph of $G$ induced by the vertices of even degree. Building on Lov\'{a}sz's result, Conjecture~\ref{conj:Gallai} has been proved for several classes of graphs defined by imposing some structure on $G_E$. The first result of this kind was obtained by Pyber.

\begin{theorem}\label{thm:Pyber}\emph{\cite{Pyber}}
If $G$ is a graph on $n$ vertices such that $G_E$ is a forest, then $G$ can be decomposed into $\fb{n}{2}$ paths.
\end{theorem}

Later, Theorem~\ref{thm:Pyber} was strengthened by Fan, who proved the following.

\begin{theorem}\label{thm:Fan}\emph{\cite{Fan}}
If $G$ is a graph on $n$ vertices such that each block of $G_E$ is a triangle free graph of maximum degree at most $3$, then $G$ can be decomposed into $\fb{n}{2}$ paths.
\end{theorem}

Gallai's conjecture is also known to hold for a variety of other graph classes. In 1988, Favaron and Koudier~\cite{eulerian} proved that the conjecture holds for graphs where the degree of every vertex is either $2$ or $4$. More recently, Botler and Jim\'{e}nez~\cite{2kReg} proved that the conjecture holds for $2k$-regular graphs of large girth and admitting a pair of disjoint perfect matchings. Jim\'{e}nez and Wakabayashi~\cite{triangleFree} showed that the conjecture holds for a subclass of planar, triangle-free graphs satisfying a distance condition on the vertices of odd degree. Finally, it was shown by Geng, Fang and Li~\cite{outerplanar}, that the conjecture holds for maximal outerplanar graphs. In this article, we prove that Gallai's conjecture holds for the class of graphs with maximum degree at most $5$.

\begin{theorem}\label{th:main}
Let $G$ be a connected graph on $n$ vertices. If $\Delta(G) \leq 5$, then $G$ admits a path decomposition into $\lceil \frac{n}{2} \rceil$ paths.
\end{theorem}

To prove Theorem~\ref{th:main}, we show that if $G$ is a smallest counterexample, then $G$ cannot contain one of $5$ configurations. This restriction is enough to show that $G_E$ is a forest, whence the result follows by Theorem~\ref{thm:Pyber}. It seems that proving Theorem~\ref{th:main} for graphs of maximum degree $6$ will require some new ideas. However, we think the approach of considering graphs of bounded maximum degree allows step-by-step improvements which could eventually lead to a general solution.

In proving special cases of Conjecture~\ref{conj:Gallai}, the presence of a ceiling in the bound brings with it a number of technical complications. It is therefore tempting to explore ways of proving a stronger, ceiling-free version except in a few special cases. We say a graph is an \emph{odd semi-clique} if it is obtained from a clique on $2k+1$ vertices by deleting at most $k-1$ edges. By a simple counting argument, we can see that an odd semi-clique on $2k+1$ vertices does not admit a path decomposition into $k$ paths. It is natural to ask if these are the only obstructions:

\begin{question}\label{conj:tight}
Does every connected graph $G$ that is not an odd semi-clique admit a path decomposition into $\lfloor \frac{|V(G)|}{2} \rfloor$ paths?
\end{question}

\section{Definitions and notation}
All graphs in this article are finite and simple, that is they contain no loops or multiple edges. We say that a path decomposition $\mathcal{D}$ of a graph $G$ is \emph{good} if $|\mathcal{D}| \leq \cb{|V(G)|}{2}$.

In figures we make use of the following conventions: Solid black circles denote vertices for which all incident edges are depicted. White hollow circles denote vertices which may have other, undepicted incident edges. Vertices containing a number indicate a vertex of that specific degree. A dotted line between two vertices indicates that those vertices are non-adjacent.

We will often modify a path decomposition of a graph $G$ to give a path decomposition of another graph $G'$. To describe these modifications we use a number of fixed expressions, which we formally define here. Let $\mathcal{D}$ be a path decomposition of $G$. Let $P \in \mathcal{D}$ be a path and $Q$ be a subpath of $P$. If $R$ is a path in $G'$ with the same end vertices as $P$, we say that we \emph{replace} $Q$ with $R$ to mean that we define a new path $P' = P - Q + R$ and redefine $\mathcal{D}$ to be the collection $\mathcal{D} - P + P'$. If $R$ is a path in $G'$ with an endpoint in common with $P$, we say that we \emph{extend} $P$ with $R$ to mean that we define a new path $P' = P + R$ and redefine $\mathcal{D}$ to be the collection $\mathcal{D} - P + P'$. For a vertex $u$ on $P$, we say that we \emph{split $P$ at $u$} to mean that we define paths $P_1$ and $P_2$ such that $P_1 \cup P_2  = P$ and $P_1 \cap P_2 = u$, and redefine $\mathcal{D}$ to be the collection $\mathcal{D} - P + P_1+P_2$. Finally, for a path $R$ in $G'$, we say that we \emph{add} the path $R$ to mean that we redefine $\mathcal{D}$ to be the collection $\mathcal{D} +R$.

\begin{proposition}\label{prop:2v}
Let $G$ and $G'$ be two graphs such that $|V(G)|\geq |V(G')|+2$, and let $\mathcal{D}$ be a path decomposition of $G$. If there is a good path decomposition $\mathcal{D'}$ of $G'$ and $|\mathcal{D}|\leq|\mathcal{D'}|+1$ then $\mathcal{D}$ is a good path decomposition of $G$. 
\end{proposition}

\begin{proof} Let $|V(G)| = n$. We have $|\mathcal{D}| \leq |\mathcal{D}'|+1 \leq \cb{n-2}{2} + 1 = \cb{n}{2}$. Thus $\mathcal{D}$ is a good path decomposition of $G$.
\end{proof}

\begin{proposition}\label{prop:0v}
Let $G$, $G_1$ and $G_2$ be three graphs such that $|V(G)|\geq |V(G_1)|+|V(G_2)|$, and let $\mathcal{D}$ be a path decomposition of $G$. If there are good path decompositions $\mathcal{D}_1$ and $\mathcal{D}_2$ of $G_1$ and $G_2$ (respectively) and $|\mathcal{D}|\leq|\mathcal{D}_1|+|\mathcal{D}_2|-1$, then $\mathcal{D}$ is a good path decomposition of $G$. 
\end{proposition}
\begin{proof} Let $G, G_1$ and $G_2$ have $n, n_1$ and $n_2$ vertices respectively. We have $|\mathcal{D}| \leq |\mathcal{D}_1| +|\mathcal{D}_2|- 1 \leq \cb{n_1}{2} + \cb{n_2}{2} - 1 \leq \cb{n}{2}$. Thus $\mathcal{D}$ is a good path decomposition of $G$.
\end{proof}

\begin{proposition}\label{prop:1v}
Let $G$, $G_1$ and $G_2$ be three graphs such that $|V(G)|\geq |V(G_1)|+|V(G_2)|+1$, and let $\mathcal{D}$ be a path decomposition of $G$. If there are good path decompositions $\mathcal{D}_1$ and $\mathcal{D}_2$ of $G_1$ and $G_2$ (respectively) and $|\mathcal{D}|\leq|\mathcal{D}_1|+|\mathcal{D}_2|$, then $\mathcal{D}$ is a good path decomposition of $G$. 
\end{proposition}
\begin{proof} Let $G, G_1$ and $G_2$ have $n, n_1$ and $n_2$ vertices respectively. We have $|\mathcal{D}| \leq |\mathcal{D}_1| +|\mathcal{D}_2| \leq \cb{n_1}{2} + \cb{n_2}{2} \leq \cb{n}{2}$. Thus $\mathcal{D}$ is a good path decomposition of $G$.
\end{proof}

\section{Main Result}

Let $G$ be a graph with $\Delta(G) \leq k$. We first prove that a number of configurations are reducible in $G$, if Gallai's conjecture holds for all smaller graphs of maximum degree $k$.

\begin{lemma}\label{lem:ReducibleConfigs}
Let $k \in \mathbb{N}$. Let $G$ be a connected graph with maximum degree $\Delta(G) \leq k$, and suppose that $G$ does not admit a good path decomposition. If $G$ is vertex minimal with these properties, then $G$ does not contain any of the following configurations (see Figure~\ref{fig:Evertex}):
\begin{enumerate}[\normalfont$C_1$:]
    \item\label{c:v2} A vertex of degree $2$ whose neighbours are not adjacent.
    \item\label{c:cutEdge} A cut-edge $uv$ such that $d(u)$ and $d(v)$ are even.
    \item\label{c:diam} An edge $uv$ such that $u$ and $v$ have precisely $2$ common neighbours, and $d(u)=d(v)=4$.
    \item\label{c:44} An edge $uv$ such that $d(u)=d(v)=4$, and for $t_1,t_2,t_3$ (resp. $w_1,w_2,w_3$) the three other neighbors of $u$ (resp. $v$), the pairs $t_1t_2$ and $w_1w_2$ are not edges and $t_3\neq w_3$.
    \item\label{c:tr} A triangle $uvw$ such that $d(u)=4$ and $d(v),d(w) \in \{2,4\}$.

\end{enumerate}
\end{lemma}

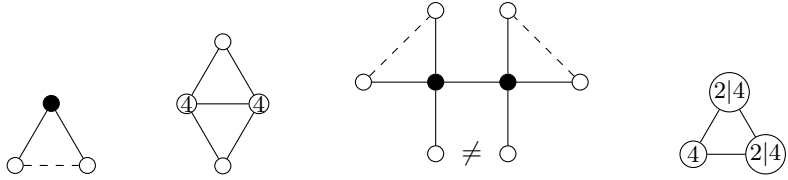
\begin{figure}[!ht]
\centering
\subfloat{
\begin{tikzpicture}[scale=0.95]
\tikzstyle{whitenode}=[draw,circle,fill=white,minimum size=6pt,inner sep=0pt]
\tikzstyle{blacknode}=[draw,circle,fill=black,minimum size=6pt,inner sep=0pt]
\draw (0,0) node[whitenode] (a1) {};
\draw (a1)
--++(60:1) node[blacknode] (a2) {}
--++(-60:1) node[whitenode] (a3) {};
\draw[dashed] (a1) edge node {} (a3);
\end{tikzpicture}
\label{fig:c:v2}
}
\qquad
\subfloat{
\centering
\begin{tikzpicture}[scale=0.95]
\tikzstyle{whitenode}=[draw,circle,fill=white,minimum size=6pt,inner sep=0pt]
\tikzstyle{blacknode}=[draw,circle,fill=black,minimum size=6pt,inner sep=0pt]
\draw (0,0) node[whitenode] (a1) {\small{$4$}};
\draw (a1)
--++(60:1) node[whitenode] (a2) {}
--++(-60:1) node[whitenode] (a3) {\small{$4$}}
--++(-120:1) node[whitenode] (a4) {};
\draw (a1) edge node {} (a3);
\draw (a1) edge node {} (a4);
\end{tikzpicture}
\label{fig:c:diam}
}
\hspace{22pt}
\subfloat{
\begin{tikzpicture}[scale=0.95]
\tikzstyle{whitenode}=[draw,circle,fill=white,minimum size=6pt,inner sep=0pt]
\tikzstyle{blacknode}=[draw,circle,fill=black,minimum size=6pt,inner sep=0pt]
\draw (0,0) node[whitenode] (a1) {}
--++(0:1) node[blacknode] (a2) {}
--++(0:1) node[blacknode] (a3) {}
--++(0:1) node[whitenode] (a4) {};
\draw(a2)
--++(-90:1) node[whitenode] (a2a) {};
\draw(a2)
--++(90:1) node[whitenode] (a2b) {};
\draw(a3)
--++(-90:1) node[whitenode] (a3a) {};
\draw(a3)
--++(90:1) node[whitenode] (a3b) {};
\draw[dashed] (a1) edge node {} (a2b);
\draw[dashed] (a4) edge node {} (a3b);
\draw (1.5,-1) node (t) {$\neq$};
\end{tikzpicture}
\label{fig:44}
}
\hspace{4pt}
\qquad
\subfloat{
\begin{tikzpicture}[scale=0.95]
\tikzstyle{whitenode}=[draw,circle,fill=white,minimum size=10pt,inner sep=0pt]
\tikzstyle{blacknode}=[draw,circle,fill=black,minimum size=6pt,inner sep=0pt]
\draw (0,0) node[whitenode] (a1) {\small{$4$}};
\draw (a1)
--++(60:1) node[whitenode] (a2) {\small{$2 | 4$}}
--++(-60:1) node[whitenode] (a3) {\small{$2 | 4$}};
\draw (a1) edge node {} (a3);
\end{tikzpicture}
\label{fig:tr}
}
\caption{Configurations $C_1$, $C_3$, $C_4$ and $C_5$ from Lemma~\ref{lem:ReducibleConfigs}.}
\label{fig:Evertex}
\end{figure}

\begin{proof}

\begin{claim}\label{cl:v2}
$G$ does not contain the configuration $C_{1}$. 
\end{claim}

\begin{proofclaim}
Suppose that the claim is false. Let $u$ be the vertex of degree $2$ with $N(u) = \{v, w\}$, and let $G'$ be the graph $G - u + vw$. Since $v$ and $w$ are non-adjacent, $G'$ is a simple graph. By the minimality of $G$, we have that $G'$ admits a good path decomposition $\mathcal{D'}$. By Proposition~\ref{prop:1v}, we obtain a good path decomposition of $G$ by replacing the edge $vw$ with the path $vuw$ (see Figure~\ref{fig:v2:tr}).
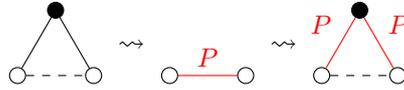
\begin{figure}[!h]
\centering
\begin{tikzpicture}[scale=1,auto]
\tikzstyle{whitenode}=[draw,circle,fill=white,minimum size=6pt,inner sep=0pt]
\tikzstyle{blacknode}=[draw,circle,fill=black,minimum size=6pt,inner sep=0pt]
\tikzstyle{ghost}=[draw=white,circle,minimum size=0pt,inner sep=0pt]
\draw (0,0) node[whitenode] (a1) {};
\draw (a1)
--++(60:1) node[blacknode] (a2) {}
--++(-60:1) node[whitenode] (a3) {};
\draw[dashed] (a1) edge node {} (a3);
\draw (1.5,0.4) node[ghost] (a1) {$\leadsto$};
\draw (2,0) node[whitenode] (a1) {};
\draw (a1)
++(60:1) node[ghost] (a2) {}
++(-60:1) node[whitenode] (a3) {};
\draw[red] (a1) edge node {$P$} (a3);
\draw (3.5,0.4) node[ghost] (a1) {$\leadsto$};
\draw (4,0) node[whitenode] (a1) {};
\draw (a1)
++(60:1) node[blacknode] (a2) {}
++(-60:1) node[whitenode] (a3) {};
\draw[red] (a1) edge node {$P$} (a2);
\draw[red] (a2) edge node {$P$} (a3);
\draw[dashed] (a1) edge node {} (a3);
\end{tikzpicture}
\caption{The reduction of $C_1$.}
\label{fig:v2:tr}
\end{figure}

This contradicts the assumption that $G$ has no such decomposition.
\end{proofclaim}

\begin{claim}\label{cl:cutEdge}
$G$ does not contain the configuration $C_2$. 
\end{claim}

\begin{proofclaim}
Suppose that the claim is false. Deleting $uv$ results in two connected graphs $G_1$ and $G_2$, containing $u$ and $v$ respectively. By the minimality of $G$, both $G_1$ and $G_2$ admit good path decompositions $\mathcal{D}_1$ and $\mathcal{D}_2$. To obtain a path decomposition of $G$, note that, since $u$ has odd degree in $G_1$, there is a path $P_u \in \mathcal{D}_1$ ending at $u$. Similarly, there is a path $P_v \in \mathcal{D}_2$ ending at $v$. Now let $\mathcal{D}$ be the path decomposition of $G$ formed by taking the union $\mathcal{D}_1 \cup \mathcal{D}_2$, deleting $P_u$ and $P_v$, and adding a new path $P = P_u+uv+P_v$ (see Figure~\ref{fig:cut:tr}). By Proposition~\ref{prop:0v}, $\mathcal{D}$ is a good path decomposition of $G$, a contradiction.

\begin{figure}[!ht]
    \centering
\begin{tikzpicture}[scale=1,auto]
\tikzstyle{whitenode}=[draw,circle,fill=white,minimum size=6pt,inner sep=0pt]
\tikzstyle{blacknode}=[draw,circle,fill=black,minimum size=6pt,inner sep=0pt]
\tikzstyle{ghost}=[draw=white,circle,minimum size=0pt,inner sep=0pt]
\draw (0,0) node[blacknode] (a) {};
\draw (a)
--++(0:1) node[blacknode] (b) {};
\draw (a)
--++(120:1) node[whitenode] (a1) {};
\draw (a)
--++(180:1) node[whitenode] (a2) {};
\draw (a)
--++(-120:1) node[whitenode] (a3) {};
\draw (b)
--++(60:1) node[whitenode] (b1) {};
\draw (b)
--++(30:1) node[whitenode] (b2) {};
\draw (b)
--++(0:1) node[whitenode] (b3) {};
\draw (b)
--++(-30:1) node[whitenode] (b4) {};
\draw (b)
--++(-60:1) node[whitenode] (b5) {};

\draw (2.5,0) node[ghost] (a1) {$\leadsto$};

\draw (4,0) node[blacknode] (a) {};
\draw (a)
++(0:1) node[blacknode] (b) {};
\draw (a)
--++(120:1) node[whitenode] (a1) {};
\draw (a)
++(180:1) node[whitenode] (a2) {};
\draw (a)
--++(-120:1) node[whitenode] (a3) {};
\draw (b)
++(60:1) node[whitenode] (b1) {};
\draw (b)
--++(30:1) node[whitenode] (b2) {};
\draw (b)
--++(0:1) node[whitenode] (b3) {};
\draw (b)
--++(-30:1) node[whitenode] (b4) {};
\draw (b)
--++(-60:1) node[whitenode] (b5) {};
\draw[red,thick] (a2) edge node {$P_u$} (a);
\draw[blue,thick] (b) edge node {$P_v$} (b1);

\draw (6.5,0) node[ghost] (a1) {$\leadsto$};
\draw (8,0) node[whitenode] (a) {};
\draw (a)
++(0:1) node[blacknode] (b) {};
\draw (a)
--++(120:1) node[whitenode] (a1) {};
\draw (a)
++(180:1) node[whitenode] (a2) {};
\draw (a)
--++(-120:1) node[whitenode] (a3) {};
\draw (b)
++(60:1) node[whitenode] (b1) {};
\draw (b)
--++(30:1) node[whitenode] (b2) {};
\draw (b)
--++(0:1) node[whitenode] (b3) {};
\draw (b)
--++(-30:1) node[whitenode] (b4) {};
\draw (b)
--++(-60:1) node[whitenode] (b5) {};
\draw[blue!50!red,thick] (a2) edge node {$P$} (a);
\draw[blue!50!red,thick] (b) edge node {$P$} (b1);
\draw[blue!50!red,thick] (a) edge node {$P$} (b);
\end{tikzpicture}
    \caption{The reduction of $C_2$.}
    \label{fig:cut:tr}
\end{figure}
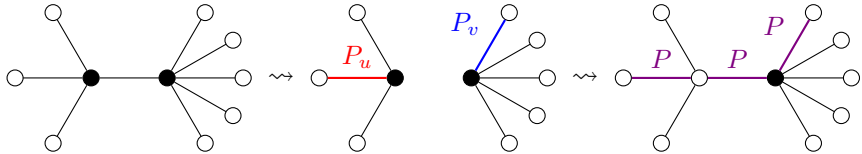
\end{proofclaim}

\begin{claim}\label{cl:diam}
$G$ does not contain the configuration $C_3$. 
\end{claim}

\begin{proofclaim}
Suppose that the claim is false. Let $x$ and $y$ be the common neighbours of $u$ and $v$. Since $d(u) = d(v) = 4$, both $u$ and $v$ both have precisely one other neighbour. Let these vertices be $u'$ and $v'$ respectively. Since $u$ and $v$ have precisely two common neighbours, we have that $u' \neq v'$. Suppose first that at most one of the edges $xu', u'y, yv', v'x$ is present in $G$, say $xu'$. If $G - u -v$ is connected, then let $G'$ be the graph $G - u - v + u'y + v'x$. Otherwise let $G' = G - u - v + u'y + v'x + xy$. Note that $G'$ is connected, and so by the minimality of $G$, it admits a good path decomposition. Now, replace $v'x$ by $xvv'$ and replace $u'y$ by $u'uy$. Furthermore, if $xy \in E(G')\setminus E(G)$, then replace $xy$ by $xuvy$. Otherwise add a new path $xuvy$ to the decomposition. By Proposition~\ref{prop:2v}, and since we add at most one new path, the resulting decomposition is a good path decomposition of $G$. This contradicts the assumption that $G$ has no such decomposition.
    
Next, suppose that $xu', u'y, yv', v'x \in E(G)$, so the graph $G' = G - u-v$ is connected. By the minimality of $G$, the graph $G'$ has a good path decomposition. Now replace the edge $xu'$ with the path $xvuu'$, and add a new path $u'xuyvv'$ to the decomposition. By Proposition~\ref{prop:2v}, and since we add at most one new path, the resulting decomposition is a good path decomposition of $G$, contradicting the assumption.
    
Finally, suppose that precisely two or three of the edges $xu', u'y, yv', v'x$ are present in $G$. As a consequence, from the set $\{xu', u'y, yv', v'x\} \setminus E(G)$, we may choose an edge, $xu'$ say, such that the graph $G' = G - u - v + xu'$ is connected. By the minimality of $G$, the graph $G'$ has a good path decomposition. Now replace $xu'$ by $xvuu'$, and add a new path $xuyvv'$ to the decomposition. Again, by Proposition~\ref{prop:2v}, and since we add at most one new path, the resulting decomposition is a good path decomposition of $G$, contradicting the assumption.

\end{proofclaim}

\begin{claim}\label{cl:44}
$G$ does not contain the configuration $C_4$. 
\end{claim}

\begin{proofclaim}
Suppose that the claim is false. Since $G$ does not contain Configuration $C_3$, the vertices $u$ and $v$ do not have precisely two common neighbours. First suppose that $u$ and $v$ have $3$ common neighbours $x,y$ and $z$. In this case, since there is a pair of non-adjacent vertices amongst $N(u) \setminus \{ v \}$, we may assume $xy \not\in E(G)$. Furthermore, by the definition of Configuration $C_4$, the third vertex $z$ is non-adjacent to at least one of $x$ or $y$. We conclude that there are two non-edges amongst $x,y$ and $z$, say these are $xy$ and $yz$. Let $G'$ be the graph $G - u - v + xy + yz$. It is easy to see that $G'$ is connected. By the minimality of $G$, the graph $G'$ has a good path decomposition. In this decomposition, replace $xy$ by $xuy$ and replace $yz$ by $yvz$. Finally, add a new path $xvuz$ (see Figure~\ref{fig:44:tr0}). This gives a good path decomposition of $G$, a contradiction. We may thus assume that $u$ and $v$ have at most one common neighbour.

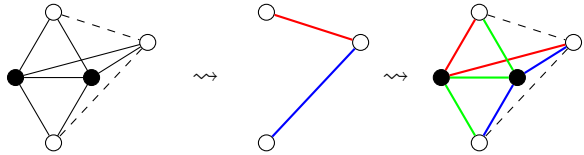
\begin{figure}[!ht]
    \centering
\begin{tikzpicture}[scale=1,auto]
\tikzstyle{whitenode}=[draw,circle,fill=white,minimum size=6pt,inner sep=0pt]
\tikzstyle{blacknode}=[draw,circle,fill=black,minimum size=6pt,inner sep=0pt]
\tikzstyle{ghost}=[draw=white,circle,minimum size=0pt,inner sep=0pt]
\draw (0,0) node[blacknode] (a1) {};
\draw (a1)
--++(60:1) node[whitenode] (a2) {}
--++(-60:1) node[blacknode] (a3) {}
--++(-120:1) node[whitenode] (a4) {};
\draw (a1) edge node {} (a3);
\draw (a1) edge node {} (a4);
\draw (a1) --++(15:1.8) node[whitenode] (a5) {};
\draw (a5) -- (a3);
\draw[dashed] (a2) -- (a5) -- (a4);

\draw (2.5,0) node[ghost] (a1) {$\leadsto$};

\draw (2.8,0) node (a1) {};
\draw (a1)
++(60:1) node[whitenode] (a2) {}
++(-60:1) node (a3) {}
++(-120:1) node[whitenode] (a4) {};
\draw (a1) ++(15:1.8) node[whitenode] (a5) {};
\draw[red,thick] (a2) -- (a5);
\draw[blue,thick] (a5)-- (a4);
\draw (5,0) node[ghost] (a1) {$\leadsto$};
\draw (5.6,0) node[blacknode] (a1) {};
\draw (a1)
++(60:1) node[whitenode] (a2) {}
++(-60:1) node[blacknode] (a3) {}
++(-120:1) node[whitenode] (a4) {};
\draw (a1) ++(15:1.8) node[whitenode] (a5) {};
\draw[thick,red] (a2) -- (a1)-- (a5);
\draw[blue,thick] (a5)-- (a3) -- (a4);
\draw[green,thick] (a2)-- (a3) -- (a1) -- (a4);
\draw[dashed] (a2) -- (a5) -- (a4);
\end{tikzpicture}
    \caption{The reduction of $C_4$ in the case where $u$ and $v$ have three common neighbors.}
    \label{fig:44:tr0}
\end{figure}

We now consider three cases depending on the structure of $G-\{u,v\}$. In each case we assume the previous ones do not apply (up to symmetry).

    \begin{enumerate}
        \item \emph{Assume that $G-u$ has at least three connected components}. Because $uv$ is not a cut-edge, the component of $G-u$ containing $v$ contains at least one other neighbor of $u$. Thus $G-u$ has precisely three components, and $t_1$ and $t_2$ lie in different components of $G-u$. Let $G'$ be the graph formed from $G-u$ by adding the edge $t_1t_2$. Thus $G'$ has two components $G_1$ and $G_2$, and by the minimality of $G$, both have good path decompositions $\mathcal{D}_1$ and $\mathcal{D}_2$. Without loss of generality we suppose $G_2$ contains $v$. Let $P \in \mathcal{D}_1$ be the path containing the edge $t_1t_2$. Furthermore, let $P_1$ and $P_2$ be the possibly empty subpaths of $P - t_1t_2$ containing $t_1$ and $t_2$ respectively. Note that since $v$ has degree $3$ in $G'$, there is some path $Q \in \mathcal{D}_2$ which ends at $v$. We construct a path decomposition of $G$ by taking the union $\mathcal{D}_1 \cup \mathcal{D}_2$ and replacing $P$ and $Q$ with the paths $P_1 + t_1uv + Q$ and $P_2+t_2ut_3$. By Proposition~\ref{prop:1v}, and since we introduced no new paths, the resulting path decomposition is good, a contradiction.

        \item \emph{Assume that $G-\{u,v\}$ has at least four connected components}.  Since both $G-u$ and $G-v$ have at most two connected components, there are precisely four connected components $C_1, C_2, C_3$ and $C_4$. Furthermore, two of these components contain both a neighbour of $u$ and a neighbour of $v$, one component contains only a neighbour of $u$, and one component contains only a neighbour of $v$. Relabeling if necessary, we may suppose that $t_1, w_1 \in C_1$, $t_2, w_2 \in C_2$, $t_3 \in C_3$ and $w_3 \in C_4$. This relabelling preserves the fact that $t_1t_2, w_1w_2 \not \in E(G)$ and $t_3 \neq w_3$. Consider the graph $G_1$ obtained from $C_1$ and $C_2$ by adding the edges $t_1t_2$ and $w_1w_2$. Similarly, consider the graph $G_2$ obtained from $C_3$ and $C_4$ by adding the edge $t_3w_3$. By the minimality of $G$, we obtain good path decompositions of $G_1$ and $G_2$, which we merge in the obvious way. The edge $t_1t_2$ is replaced with $t_1ut_2$, $w_1w_2$ with $w_1vw_3$, and $t_3w_3$ with $t_3uvw_3$) to obtain a path decomposition of $G$. By Proposition~\ref{prop:1v}, this yields a good path decomposition of $G$.
\item Now $G-\{u,v\}$ has at most three connected components, and each of $G-u$ and $G-v$ has at most two connected components. Let $T = \{t_1,t_2,t_3\}$ and $W = \{w_1,w_2,w_3\}$. We claim that we can relabel the vertices in $T$ and $W$ such that the graph $G-u-v + t_1t_2 + w_1w_2$ is connected and the properties that $t_1t_2, w_1w_2 \not \in E(G)$ and $t_3 \neq w_3$ are preserved. Indeed if $u$ and $v$ have a common neighbour, let $t \in T$ and $w \in W$ be such that $t=w$. Otherwise let $t=t_1$ and $w = w_1$. Suppose first that $t$ and $w$ lie in the same component of $G - u - v$.  Since $G-u-v$ has at most $3$ components, and $G - u $ and $G-v$ have at most $2$ components, there are non edges $tt'$ and $ww'$ for some $t' \in T$ and $w' \in W$ such that $G - u - v + tt' + ww'$ is connected. Furthermore, since $t$ and $w$ are the only possible common neighbours of $u$ and $v$, we have that the single vertices in $T \setminus \{t,t'\}$ and $W \setminus \{w,w'\}$ are not equal. Thus, letting $t_1 = t$, $t_2 = t'$, $w_1 = w$, $w_2 = w'$ and setting $t_3$ and $w_3$ to be the remaining vertices gives the desired relabeling.

 Suppose now that $t$ and $w$ lie in different components of $G-u-v$. In particular this implies that $T \cap W = \emptyset$. Again, since $G-u-v$ has at most $3$ components, and $G - u $ and $G-v$ have at most $2$ components, there are non-edges $e_T$ and $e_W$ amongst the vertices of $T$ and $W$ respectively, such that  $G - u - v + e_T + e_W$ is connected. We relabel the vertices in $T$ and $W$ such that $t_1$ and $t_2$ are the endpoints of $e_T$, $w_1$ and $w_2$ are the endpoints of $e_W$, and $t_3$ and $w_3$ are the remaining vertices. Since $T \cap W = \emptyset$, we have that $t_3 \neq w_3$ are required.

Let $G'$ be the graph obtained from $G-\{u,v\}$ by adding the edges $t_1t_2$ and $w_1w_2$. By the argument above, $G'$ is connected, and so by the minimality of $G$, there is a good path decomposition of $G'$. We obtain a path decomposition of $G$ by replacing $t_1t_2$ with $t_1ut_2$ and $w_1w_2$ with $w_1vw_2$, and adding the path $t_3uvw_3$. Note that since $t_3 \neq w_3$ the latter is really a path. By Proposition~\ref{prop:2v}, and since we add at most one new path, this yields a good path decomposition of $G$.

 \end{enumerate}
    
    \begin{figure}[!ht]
        \centering
        \begin{tikzpicture}[scale=1,auto]
\tikzstyle{whitenode}=[draw,circle,fill=white,minimum size=6pt,inner sep=0pt]
\tikzstyle{blacknode}=[draw,circle,fill=black,minimum size=6pt,inner sep=0pt]
\tikzstyle{ghost}=[draw=white,circle,minimum size=0pt,inner sep=0pt]
\draw (0,0) node[whitenode] (a1) {}
--++(0:1) node[blacknode] (a2) {}
--++(0:1) node[blacknode] (a3) {}
--++(0:1) node[whitenode] (a4) {};
\draw(a2)
--++(-90:1) node[whitenode] (a2a) {};
\draw(a2)
--++(90:1) node[whitenode] (a2b) {};
\draw(a3)
--++(-90:1) node[whitenode] (a3a) {};
\draw(a3)
--++(90:1) node[whitenode] (a3b) {};
\draw[dashed] (a1) edge node {} (a2b);
\draw[dashed] (a4) edge node {} (a3b);

\draw (3.5,0.4) node[ghost] (a1) {$\leadsto$};
\draw (4,0) node[whitenode] (a1) {}
++(0:1) node[ghost] (a2) {}
++(0:1) node[ghost] (a3) {}
++(0:1) node[whitenode] (a4) {};
\draw(a2)
++(-90:1) node[whitenode] (a2a) {};
\draw(a2)
++(90:1) node[whitenode] (a2b) {};
\draw(a3)
++(-90:1) node[whitenode] (a3a) {};
\draw(a3)
++(90:1) node[whitenode] (a3b) {};
\draw[red] (a1) edge node {$P$} (a2b);
\draw[green] (a3b) edge node {$Q$} (a4);

\draw (7.5,0.4) node[ghost] (a1) {$\leadsto$};
\draw (8,0) node[whitenode] (a1) {}
++(0:1) node[blacknode] (a2) {}
++(0:1) node[blacknode] (a3) {}
++(0:1) node[whitenode] (a4) {};
\draw(a2)
++(-90:1) node[whitenode] (a2a) {};
\draw(a2)
++(90:1) node[whitenode] (a2b) {};
\draw(a3)
++(-90:1) node[whitenode] (a3a) {};
\draw(a3)
++(90:1) node[whitenode] (a3b) {};
\draw[dashed] (a1) edge node {} (a2b);
\draw[dashed] (a4) edge node {} (a3b);
\draw[red] (a1) edge node {$P$} (a2);
\draw[red] (a2) edge node {$P$} (a2b);
\draw[green] (a3) edge node {$Q$} (a4);
\draw[green] (a3b) edge node {$Q$} (a3);
\draw[blue] (a2a) edge node {$R$} (a2);
\draw[blue] (a3) edge node {$R$} (a2);
\draw[blue] (a3) edge node {$R$} (a3a);
\end{tikzpicture}
        \caption{The reduction of $C_4$ in the connected case.}
        \label{fig:44:tr}
    \end{figure}
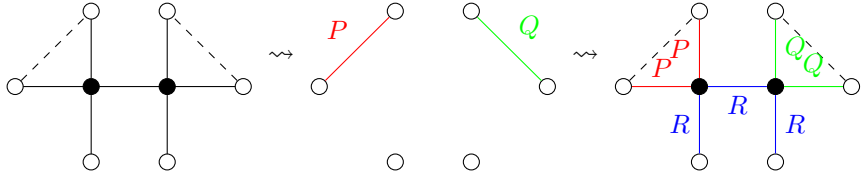
\end{proofclaim}

\begin{claim}\label{cl:tr}
$G$ does not contain the configuration $C_5$. 
\end{claim}

\begin{proofclaim}
We first consider the case where a pair in $\{u,v,w\}$, say $\{u,v\}$, has three common neighbors. Let $x$ and $y$ be the two neighbors of $\{u,v\}$ besides $w$. We argue that $wxy$ induces a triangle. Indeed, first assume there are at least two edges missing, say $xw, wy \not\in E(G)$. Consider the graph $G''=G+xw+wy$, note that it is connected, and consider a good path decomposition of it. We obtain a path decomposition of $G$ by replacing the edge $xw$ with $xuw$, replacing the edge $wy$ with $wvy$, and adding the path $xvuy$, see Figure~\ref{fig:tr1}. By Proposition~\ref{prop:2v}, this yields a good path decomposition of $G$.

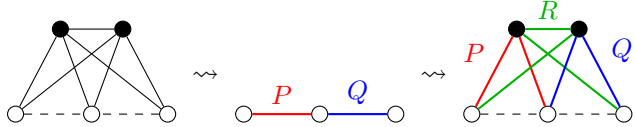
\begin{figure}[!ht]
    \centering
\begin{tikzpicture}[scale=1,auto]
\tikzstyle{whitenode}=[draw,circle,fill=white,minimum size=6pt,inner sep=0pt]
\tikzstyle{blacknode}=[draw,circle,fill=black,minimum size=6pt,inner sep=0pt]
\tikzstyle{ghost}=[draw=white,circle,minimum size=0pt,inner sep=0pt]
\draw (0,0) node[whitenode] (a) {};
\draw (a)
++(0:1) node[whitenode] (b) {}
++(0:1) node[whitenode] (c) {};
\draw (b)
++(70:1.2) node[blacknode] (d1) {};
\draw (b)
++(70+40:1.2) node[blacknode] (d2) {};

\draw[dashed] (a) edge node {} (b);
\draw[dashed] (c) edge node {} (b);
\draw (a) edge node {} (d1);
\draw (b) edge node {} (d1);
\draw (c) edge node {} (d1);
\draw (a) edge node {} (d2);
\draw (b) edge node {} (d2);
\draw (c) edge node {} (d2);
\draw (d1) edge node {} (d2);

\draw (2.5,0.5) node[ghost] (a1) {$\leadsto$};

\draw (3,0) node[whitenode] (a) {};
\draw (a)
++(0:1) node[whitenode] (b) {}
++(0:1) node[whitenode] (c) {};

\draw[thick,red] (a) edge node {$P$} (b);
\draw[thick,blue] (b) edge node {$Q$} (c);

\draw (5.5,0.5) node[ghost] (a1) {$\leadsto$};

\draw (6,0) node[whitenode] (a) {};
\draw (a)
++(0:1) node[whitenode] (b) {}
++(0:1) node[whitenode] (c) {};
\draw (b)
++(70:1.2) node[blacknode] (d1) {};
\draw (b)
++(70+40:1.2) node[blacknode] (d2) {};

\draw[dashed] (a) edge node {} (b);
\draw[dashed] (c) edge node {} (b);
\draw[thick,red] (a) edge node {$P$} (d2);
\draw[thick,red] (b) edge node {} (d2);
\draw[thick,green!70!black] (c) edge node {} (d2);
\draw[thick,green!70!black] (a) edge node {} (d1);
\draw[thick,blue] (b) edge node {} (d1);
\draw[thick,blue] (d1) edge node {$Q$} (c);
\draw[thick,green!70!black] (d2) edge node {$R$} (d1);

\end{tikzpicture}
    \caption{The reduction of $C_5$ when $u$ and $v$ have three common neighbors that induce at least two non-edges.}
    \label{fig:tr1}
\end{figure}

Assume now that there is precisely one edge missing, say the edge $xy$. Consider $G'$, the graph obtained from $G-\{u,v\}$ by adding the edge $xy$. If $G'$ is connected, then by the minimality of $G$, it has a good path decomposition. From this, we obtain a path decomposition of $G$ by replacing the edge $xy$ with $xuvy$ and adding the path $xvwuy$, see Figure~\ref{fig:tr0}. By Proposition~\ref{prop:2v}, this yields a good path decomposition of $G$. 

\begin{figure}[!ht]
    \centering
\begin{tikzpicture}[scale=1,auto]
\tikzstyle{whitenode}=[draw,circle,fill=white,minimum size=6pt,inner sep=0pt]
\tikzstyle{blacknode}=[draw,circle,fill=black,minimum size=6pt,inner sep=0pt]
\tikzstyle{ghost}=[draw=white,circle,minimum size=0pt,inner sep=0pt]
\draw (0,0) node[whitenode] (a) {};
\draw (a)
++(0:1) node[whitenode] (b) {}
++(0:1) node[whitenode] (c) {};
\draw (b)
++(70:1.2) node[blacknode] (d1) {};
\draw (b)
++(70+40:1.2) node[blacknode] (d2) {};

\draw[dashed] (a) edge node {} (b);
\draw (c) edge node {} (b);
\draw[bend right] (a) edge node {} (c);
\draw (a) edge node {} (d1);
\draw (b) edge node {} (d1);
\draw (c) edge node {} (d1);
\draw (a) edge node {} (d2);
\draw (b) edge node {} (d2);
\draw (c) edge node {} (d2);
\draw (d1) edge node {} (d2);

\draw (2.5,0.5) node[ghost] (a1) {$\leadsto$};

\draw (3,0) node[whitenode] (a) {};
\draw (a)
++(0:1) node[whitenode] (b) {}
++(0:1) node[whitenode] (c) {};

\draw[thick,red] (a) edge node {$P$} (b);
\draw (c) edge node {} (b);
\draw[bend right] (a) edge node {} (c);

\draw (5.5,0.5) node[ghost] (a1) {$\leadsto$};

\draw (6,0) node[whitenode] (a) {};
\draw (a)
++(0:1) node[whitenode] (b) {}
++(0:1) node[whitenode] (c) {};
\draw (b)
++(70:1.2) node[blacknode] (d1) {};
\draw (b)
++(70+40:1.2) node[blacknode] (d2) {};

\draw[dashed] (a) edge node {} (b);
\draw[thick,red] (a) edge node {$P$} (d2);
\draw[thick,red] (d1) edge node {} (d2);
\draw[thick,red] (b) edge node {} (d1);
\draw[thick,green!70!black] (c) edge node {} (d2);
\draw[thick,green!70!black] (d1) edge node {$Q$} (c);
\draw[thick,green!70!black] (d1) edge node {} (a);
\draw[thick,green!70!black] (d2) edge node {} (b);
\draw (c) edge node {} (b);
\draw[bend right] (a) edge node {} (c);

\end{tikzpicture}
    \caption{The reduction of $C_5$ when $u$ and $v$ have three common neighbors that induce precisely one non-edge.}
    \label{fig:tr0}
\end{figure}
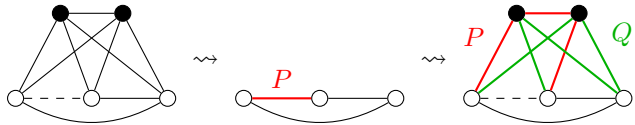

Therefore $x,y$ and $w$ induce a triangle. Let $G'=G-\{u,v\}$, and note that $G'$ is connected. Thus, by the minimality of $G$, the graph $G'$ admits a good path decomposition $\mathcal{D}'$. We obtain a path decomposition of $G$ as follows: First assume without loss of generality that $xy$ and $wy$ do not belong to the same path of $\mathcal{D}'$. Let $Q'$ be the path of $\mathcal{D}'$ containing the edge $xw$, and $Q=Q' - xw + xu$. We consider $\mathcal{D}''= \mathcal{D}'-Q'+Q$. Let $P'$ be the path of $\mathcal{D}''$ containing the edge $xy$. We write $P'=P'_1+xy+P'_2$, where $P'_1$ and $P'_2$ may be empty paths. Set $P_1=P'_1+xyvw$ and $P_2=P'_2+yuvxw$. Note that $P_1$ and $P_2$ are paths even if $P'_1$ also contains the edge $xu$ or in other words $P'_1 = Q$. Finally, let $R'$ be the path of $\mathcal{D}''$ containing the edge $yw$, and set $R=R'+wu$. We note that $\mathcal{D} = \mathcal{D}''-P'-R' +P_1+P_2+R$ is a path decomposition of $G$, with precisely one more path than $\mathcal{D}'$, see Figure~\ref{fig:tr2}. Thus $\mathcal{D}$ is a good path decomposition by Proposition~\ref{prop:2v}, a contradiction. Therefore no pair in $\{u,v,w\}$ has three common neighbors.

\begin{figure}[!ht]
    \centering
\begin{tikzpicture}[scale=0.9,auto]
\tikzstyle{whitenode}=[draw,circle,fill=white,minimum size=6pt,inner sep=0pt]
\tikzstyle{blacknode}=[draw,circle,fill=black,minimum size=6pt,inner sep=0pt]
\tikzstyle{ghost}=[draw=white,circle,minimum size=0pt,inner sep=0pt]
\draw (0,0) node[whitenode] (a) {};
\draw (a)
++(0:3.5) node[whitenode] (b) {};
\draw (a)
++(60:3.5) node[blacknode] (c) {};
\draw (c)
++(-70:1.75) node[blacknode] (d1) {};
\draw (c)
++(-110:1.75) node[blacknode] (d2) {};

\draw (a) edge node {} (b);
\draw[bend left] (a) edge node {} (c);
\draw[bend left] (c) edge node {} (b);
\draw (a) edge node {} (d1);
\draw (a) edge node {} (d2);
\draw (b) edge node {} (d1);
\draw (b) edge node {} (d2);
\draw (c) edge node {} (d1);
\draw (c) edge node {} (d2);
\draw (d1) edge node {} (d2);

\draw (4,1.5) node[ghost] (a1) {$\leadsto$};

\draw (4.5,0) node[whitenode] (a) {};
\draw (a)
++(0:3.5) node[whitenode] (b) {};
\draw (a)
++(60:3.5) node[blacknode] (c) {};

\draw[blue,thick] (a) edge node {$P'$} (b);
\draw[red,bend left] (a) edge node {$Q'$} (c);
\draw[orange,bend left] (c) edge node {$R'$} (b);

\draw (8.5,1.5) node[ghost] (a1) {$\leadsto$};

\draw (9,0) node[whitenode] (a) {};
\draw (a)
++(0:3.5) node[whitenode] (b) {};
\draw (a)
++(60:3.5) node[blacknode] (c) {};
\draw (c)
++(-70:1.75) node[blacknode] (d1) {};
\draw (c)
++(-110:1.75) node[blacknode] (d2) {};

\draw[blue!50!red,thick] (a) edge node {$P_1$} (b);
\draw[red,thick] (a) edge node {$Q$} (d2);
\draw[bend left,green!60!blue,thick] (a) edge node {$P_2$} (c);
\draw[bend left,orange,thick] (c) edge node {$R$} (b);
\draw[blue!50!red,thick] (d1) edge node {} (b);
\draw[green!60!blue,thick] (b) edge node {} (d2);
\draw[green!60!blue,thick] (a) edge node {} (d1);
\draw[orange,thick] (c) edge node {} (d2);
\draw[blue!50!red,thick] (d1) edge node {} (c);
\draw[green!60!blue,thick] (d1) edge node {} (d2);

\end{tikzpicture}
    \caption{The reduction of $C_5$ when $u$ and $v$ have three common neighbors that induce a triangle. We assume $P'$ and $R'$ are distinct, though $Q'$ might be the same as $R'$ or $P'$ or be altogether distinct from both.}
    \label{fig:tr2}
\end{figure}
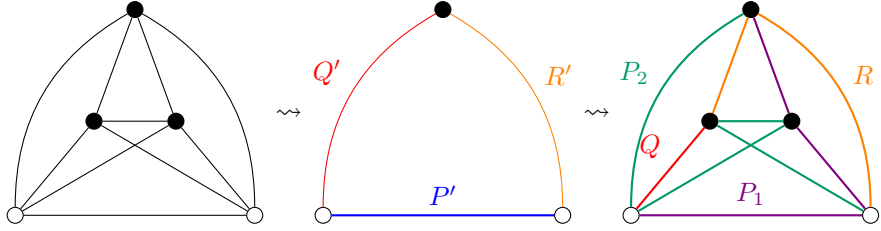

Now, since $d(u), d(v), d(w) \leq 4$ and by Claim~\ref{cl:diam}, we conclude that no pair of vertices in $\{u,v,w\}$ has a common neighbor other than the third vertex. If they exist, let $\{x_1,x_2\}$, $\{y_1,y_2\}$ and $\{z_1,z_2\}$ be the two other neighbors of $u$, $v$ and $w$ respectively. We consider three cases.
    \begin{enumerate}
        \item \emph{Assume first that one of $v$ and $w$ has degree $2$}, say $d(v)=2$. Let $G'$ be the graph obtained from $G-v$ by contracting the edge $uw$. Note that $G'$ is connected and $|V(G')| = |V(G)|-2$. By the minimality of $G$, there is a good path decomposition $\mathcal{D}'$ of $G'$. To obtain a path decomposition of $G$, we consider two cases depending on whether $ux_1$ and $ux_2$ belong to the same path in $\mathcal{D}'$, see Figures~\ref{fig:tr3a} and~\ref{fig:tr3b}. If they do not, then replace $ux_1$ with the path $wux_1$, and replace $ux_2$ with the path $wvux_2$. However, if $ux_1$ and $ux_2$ belong to the same path $P \in \mathcal{D}'$, then split $P$ at $u$ into two paths $P_1$ and $P_2$. Extend $P_1$ with the edge $uw$ and extend $P_2$ with the path $uvw$. Note that no edge incident to $w$ is in $P_1$ or $P_2$. By Proposition~\ref{prop:2v}, and since we created at most one new path, this yields a good path decomposition of $G$.

  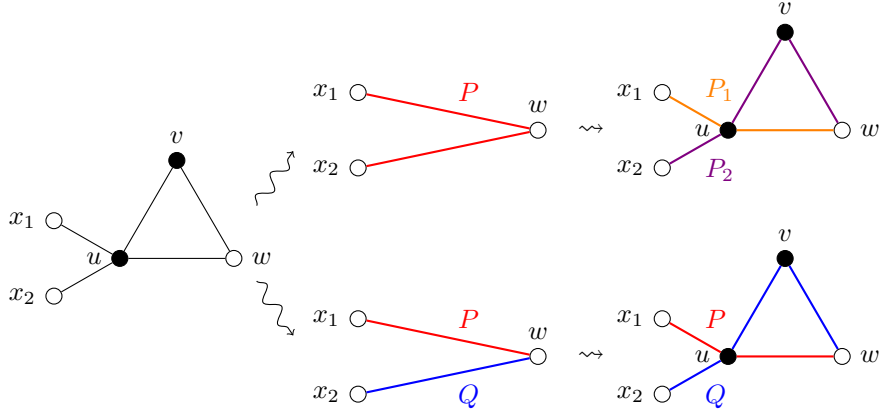
\begin{figure}[!ht]
\centering
\begin{tikzpicture}[scale=1,auto]
\tikzstyle{whitenode}=[draw,circle,fill=white,minimum size=6pt,inner sep=0pt]
\tikzstyle{blacknode}=[draw,circle,fill=black,minimum size=6pt,inner sep=0pt]
\tikzstyle{ghost}=[draw=white,circle,minimum size=0pt,inner sep=0pt]
\draw (0,-0.2) node[blacknode] (a) [label=180:$u$] {};
\draw (a)
++(0:1.5) node[whitenode] (b) [label=0:$w$] {};
\draw (a)
++(60:1.5) node[blacknode] (c) [label=90:$v$] {};
\draw (a)
++(180-30:1) node[whitenode] (d1) [label=180:$x_1$] {};
\draw (a)
++(180+30:1) node[whitenode] (d2) [label=180:$x_2$] {};

\draw (a) edge node {} (b);
\draw (a) edge node {} (c);
\draw (c) edge node {} (b);
\draw (a) edge node {} (d1);
\draw (a) edge node {} (d2);

\draw[style={decorate, decoration=snake},post] (1.8,0.5) -- (2.3,1.25);
\draw[style={decorate, decoration=snake},post] (1.8,-0.5) -- (2.3,-1.25);

\draw (4,1.5) node[ghost] (a) {};
\draw (a)
++(0:1.5) node[whitenode] (b) [label=90:$w$] {};
\draw (a)
++(180-30:1) node[whitenode] (d1) [label=180:$x_1$] {};
\draw (a)
++(180+30:1) node[whitenode] (d2) [label=180:$x_2$] {};

\draw[thick,red] (d1) edge node {$P$} (b);
\draw[thick,red] (b) edge node {} (d2);

\draw (6.2,1.5) node[ghost] (a1) {$\leadsto$};

\draw (8,1.5) node[blacknode] (a) [label=180:$u$] {};
\draw (a)
++(0:1.5) node[whitenode] (b) [label=0:$w$] {};
\draw (a)
++(60:1.5) node[blacknode] (c) [label=90:$v$] {};
\draw (a)
++(180-30:1) node[whitenode] (d1) [label=180:$x_1$] {};
\draw (a)
++(180+30:1) node[whitenode] (d2) [label=180:$x_2$] {};

\draw[orange,thick] (a) edge node {} (b);
\draw[red!50!blue,thick] (a) edge node {} (c);
\draw[red!50!blue,thick] (c) edge node {} (b);
\draw[orange,thick] (d1) edge node {$P_1$} (a);
\draw[red!50!blue,thick] (a) edge node {$P_2$} (d2);

\draw (4,-1.5) node[ghost] (a) {};
\draw (a)
++(0:1.5) node[whitenode] (b) [label=90:$w$] {};
\draw (a)
++(180-30:1) node[whitenode] (d1) [label=180:$x_1$] {};
\draw (a)
++(180+30:1) node[whitenode] (d2) [label=180:$x_2$] {};

\draw[thick,red] (d1) edge node {$P$} (b);
\draw[thick,blue] (b) edge node {$Q$} (d2);

\draw (6.2,-1.5) node[ghost] (a1) {$\leadsto$};

\draw (8,-1.5) node[blacknode] (a) [label=180:$u$] {};
\draw (a)
++(0:1.5) node[whitenode] (b) [label=0:$w$] {};
\draw (a)
++(60:1.5) node[blacknode] (c) [label=90:$v$] {};
\draw (a)
++(180-30:1) node[whitenode] (d1) [label=180:$x_1$] {};
\draw (a)
++(180+30:1) node[whitenode] (d2) [label=180:$x_2$] {};

\draw[red,thick] (a) edge node {} (b);
\draw[blue,thick] (a) edge node {} (c);
\draw[blue,thick] (c) edge node {} (b);
\draw[red,thick] (d1) edge node {$P$} (a);
\draw[blue,thick] (a) edge node {$Q$} (d2);

\end{tikzpicture}
    \caption{The reduction of $C_5$ when $u$ and $w$ have precisely one common neighbor and $d(v)=2$.}
    \label{fig:tr3a}
\end{figure}
        
        \item \emph{Assume that one of the edges $ux_1, ux_2, vy_1,vy_2, wz_1, wz_2$ is not a cut-edge}. Assume without loss of generality that $ux_1$ is such an edge. Let $G'$ be the graph obtained from $G-u$ by contracting the edge $vw$ to a vertex $s$, and adding the edge $sx_2$. Note that $G'$ is connected and $|V(G')| = |V(G)|-2$, so by the minimality of $G$, there is a good path decomposition $\mathcal{D}'$ of $G'$.

        We obtain a path decomposition of $G$ as follows. We first replace any subpath of the form $ysz$, $y \in \{y_1, y_2\}$, $z \in \{z_1,z_2\}$ with $yvwz$ (preferably) or with $yvuwz$ (if there are two such subpaths). We then replace any subpath of the form $x_2st$, $t \in \{y_1,y_2,z_1,z_2\}$, with $x_2urt$ where $r$ is the vertex of $\{v,w\}$ adjacent to $t$. We replace any remaining edge of the form $ts$, $t \in \{x_2,y_1,y_2,z_1,z_2\}$ with $tr$, where $r$ is the vertex of $\{u,v,w\}$ adjacent to $t$. Let $\mathcal{D}''$ be the resulting collection of disjoint paths in $G$. Note that since $d(s)=5$, there is a path $P$ in $\mathcal{D}'$ that ends in $s$, thus a path $P'$ in $\mathcal{D}''$ that ends in $r \in \{u,v,w\}$. We consider the set of edges of $G$ that do not belong to a path in $\mathcal{D}''$. If that set does not induce a path, then we extend $P'$ to $wu$ or $wv$. Note that this guarantees the only remaining edges induce a path $Q$, which we add to the path collection. By Proposition~\ref{prop:2v}, and since we added at most one new path, this yields a good path decomposition of $G$.
        
          \begin{figure}[!ht]
\centering
\begin{tikzpicture}[scale=1,auto]
\tikzstyle{whitenode}=[draw,circle,fill=white,minimum size=6pt,inner sep=0pt]
\tikzstyle{blacknode}=[draw,circle,fill=black,minimum size=6pt,inner sep=0pt]
\tikzstyle{ghost}=[draw=white,circle,minimum size=0pt,inner sep=0pt]
\draw (0,0) node[blacknode] (a) [label=180:$v$] {};
\draw (a)
++(0:1.5) node[blacknode] (b) [label=0:$u$] {};
\draw (a)
++(60:1.5) node[blacknode] (c) [label=90:$w$] {};
\draw (a)
++(180-30:1) node[whitenode] (d1) {};
\draw (a)
++(180+30:1) node[whitenode] (d2) {};
\draw (b)
++(-30:1) node[whitenode] (e1) {};
\draw (b)
++(30:1) node[whitenode] (e2) {};
\draw (c)
++(90-30:1) node[whitenode] (f1) {};
\draw (c)
++(90+30:1) node[whitenode] (f2) {};

\draw (a) edge node {} (b);
\draw (a) edge node {} (c);
\draw (c) edge node {} (b);
\draw (a) edge node {} (d1);
\draw (a) edge node {} (d2);
\draw (b) edge node {} (e1);
\draw (b) edge node {} (e2);
\draw (c) edge node {} (f1);
\draw (c) edge node {} (f2);

\draw[style={decorate, decoration=snake},post] (2.3,0) -- (2.8,0);

\draw (4,0) node[blacknode] (a) {};
\draw (a)
++(0:1.5) node[ghost] (b) {};
\draw (a)
++(60:1.5) node[ghost] (c) {};
\draw (a)
++(180-30:1) node[whitenode] (d1) {};
\draw (a)
++(180+30:1) node[whitenode] (d2) {};
\draw (b)
++(-30:1) node[whitenode] (e1) {};
\draw (b)
++(30:1) node[whitenode] (e2) {};
\draw (c)
++(90-30:1) node[whitenode] (f1) {};
\draw (c)
++(90+30:1) node[whitenode] (f2) {};

\draw[red,thick] (d1) edge node {} (a);
\draw[blue!20!white!70!black,very thick] (a) edge node {$P'$} (d2);
\draw[red,thick] (f1) edge node {$Q$} (a);
\draw[blue!80!black,thick] (a) edge node {} (f2);
\draw[orange,thick] (a) edge node {} (e2);

\draw[style={decorate, decoration=snake},post] (6.3,0) -- (6.8,0);

\draw (8,0) node[blacknode] (a) {};
\draw (a)
++(0:1.5) node[blacknode] (b) {};
\draw (a)
++(60:1.5) node[blacknode] (c) {};
\draw (a)
++(180-30:1) node[whitenode] (d1) {};
\draw (a)
++(180+30:1) node[whitenode] (d2) {};
\draw (b)
++(-30:1) node[whitenode] (e1) {};
\draw (b)
++(30:1) node[whitenode] (e2) {};
\draw (c)
++(90-30:1) node[whitenode] (f1) {};
\draw (c)
++(90+30:1) node[whitenode] (f2) {};

\draw[blue!20!white!70!black,very thick] (a) edge node {} (b);
\draw[red,thick] (a) edge node {} (c);
\draw[green!70!black,thick] (c) edge node {} (b);
\draw[red,thick] (a) edge node {} (d1);
\draw[blue!20!white!70!black,very thick] (a) edge node {$P'$} (d2);
\draw[green!70!black,thick] (b) edge node {} (e1);
\draw[orange,thick] (b) edge node {} (e2);
\draw[red,thick] (f1) edge node {$Q$} (c);
\draw[blue,thick] (c) edge node {} (f2);

\end{tikzpicture}
    \caption{An example of the reduction of $C_5$ when $d(u)=d(v)=d(w)=4$, the triangle $(u,v,w)$ is adjacent to no other triangle and some edge in $E(\{u,v,w\},\{x_1,x_2,y_1,y_2,z_1,z_2\})$ is not a cut-edge.}
    \label{fig:tr3b}
\end{figure}
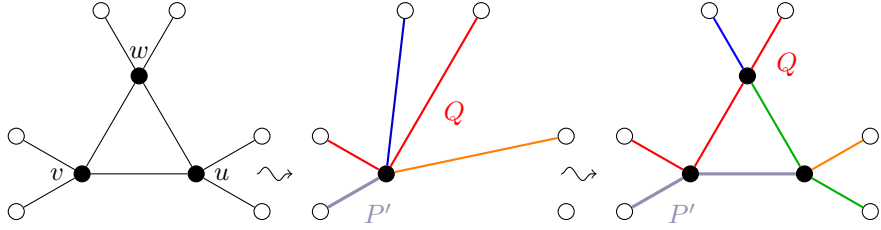
        
       \item Now $d(u)=d(v)=d(w)=4$ and every edge with precisely one endpoint in $\{u,v,w\}$ is a cut-edge. Consider the graph $G'$ obtained from $G-\{u,v,w\}$ by adding the three edges $x_1y_1$, $x_2y_2$ and $z_1z_2$. Note that $G'$ has precisely three connected components $G_1, G_2$, and $G_3$. By the minimality of $G$, there are good path decompositions of $G_1$, $G_2$ and $G_3$. We obtain a path decomposition of $G$ by replacing $x_1y_1$ with the path $x_1uvy_1$, replacing $x_2y_2$ with $x_2uwvy_2$, and replacing $z_1z_2$ with the path $z_1wz_2$ (see Figure~\ref{fig:tr4}). These paths are all distinct since the edges $x_1y_1$, $x_2y_2$ and $z_1z_2$ belong to different components of $G'$. Note that the total number of paths involved in the resulting path decomposition of $G$ is at most $\frac{|V(G_1)|+1}2+\frac{|V(G_2)|+1}2+\frac{|V(G_3)|+1}2=\frac{|V(G)|}2$, thus it is a good path decomposition.
       
       \begin{figure}[!ht]
\centering
\begin{tikzpicture}[scale=1,auto]
\tikzstyle{whitenode}=[draw,circle,fill=white,minimum size=6pt,inner sep=0pt]
\tikzstyle{blacknode}=[draw,circle,fill=black,minimum size=6pt,inner sep=0pt]
\tikzstyle{ghost}=[draw=white,circle,minimum size=0pt,inner sep=0pt]
\draw (0,0) node[blacknode] (a) [label=180:$v$] {};
\draw (a)
++(0:1.5) node[blacknode] (b) [label=0:$u$] {};
\draw (a)
++(60:1.5) node[blacknode] (c) [label=90:$w$] {};
\draw (a)
++(180-30:1) node[whitenode] (d1) {};
\draw (a)
++(180+30:1) node[whitenode] (d2) {};
\draw (b)
++(-30:1) node[whitenode] (e1) {};
\draw (b)
++(30:1) node[whitenode] (e2) {};
\draw (c)
++(90-30:1) node[whitenode] (f1) {};
\draw (c)
++(90+30:1) node[whitenode] (f2) {};

\draw (a) edge node {} (b);
\draw (a) edge node {} (c);
\draw (c) edge node {} (b);
\draw (a) edge node {} (d1);
\draw (a) edge node {} (d2);
\draw (b) edge node {} (e1);
\draw (b) edge node {} (e2);
\draw (c) edge node {} (f1);
\draw (c) edge node {} (f2);

\draw[style={decorate, decoration=snake},post] (2.3,0) -- (2.8,0);

\draw (4,0) node[ghost] (a) {};
\draw (a)
++(0:1.5) node[ghost] (b) {};
\draw (a)
++(60:1.5) node[ghost] (c) {};
\draw (a)
++(180-30:1) node[whitenode] (d1) {};
\draw (a)
++(180+30:1) node[whitenode] (d2) {};
\draw (b)
++(-30:1) node[whitenode] (e1) {};
\draw (b)
++(30:1) node[whitenode] (e2) {};
\draw (c)
++(90-30:1) node[whitenode] (f1) {};
\draw (c)
++(90+30:1) node[whitenode] (f2) {};

\draw[red,thick] (d1) edge node {} (f2);
\draw[blue,thick] (d2) edge node {} (f1);
\draw[green!80!black,thick,bend left] (e1) edge node {} (e2);

\draw[style={decorate, decoration=snake},post] (6.5,0) -- (6.99,0);

\draw (8,0) node[blacknode] (a) {};
\draw (a)
++(0:1.5) node[blacknode] (b) {};
\draw (a)
++(60:1.5) node[blacknode] (c) {};
\draw (a)
++(180-30:1) node[whitenode] (d1) {};
\draw (a)
++(180+30:1) node[whitenode] (d2) {};
\draw (b)
++(-30:1) node[whitenode] (e1) {};
\draw (b)
++(30:1) node[whitenode] (e2) {};
\draw (c)
++(90-30:1) node[whitenode] (f1) {};
\draw (c)
++(90+30:1) node[whitenode] (f2) {};

\draw[red,thick] (d1) edge node {} (a);
\draw[red,thick] (a) edge node {} (c);
\draw[red,thick] (c) edge node {} (f2);
\draw[blue,thick] (d2) edge node {} (a);
\draw[blue,thick] (a) edge node {} (b);
\draw[blue,thick] (b) edge node {} (c);
\draw[blue,thick] (f1) edge node {} (c);
\draw[green!80!black,thick] (e1) edge node {} (b);
\draw[green!80!black,thick] (b) edge node {} (e2);

\end{tikzpicture}
    \caption{The reduction of $C_5$ when $d(u)=d(v)=d(w)=4$, the triangle $(u,v,w)$ is adjacent to no other triangle and every edge in $E(\{u,v,w\},\{x_1,x_2,y_1,y_2,z_1,z_2\})$ is a cut-edge.}
    \label{fig:tr4}
\end{figure}
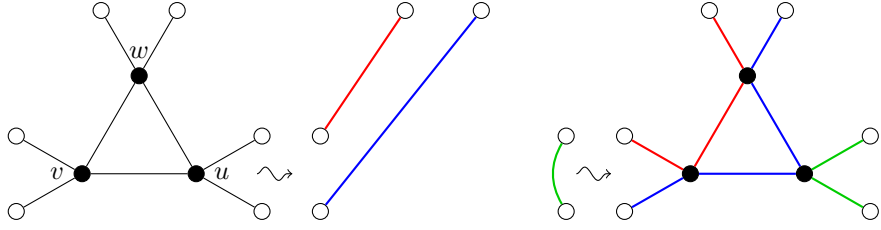
        
    \end{enumerate}
\end{proofclaim}

By Claims~\ref{cl:v2},~\ref{cl:cutEdge},~\ref{cl:diam},~\ref{cl:44} and \ref{cl:tr}, the lemma statement holds.
\end{proof}

Recall that $G_E$ denotes the graph induced on the vertices of even degree in $G$.

\begin{lemma}\label{lem:Structure}
Let $G$ be a connected graph such that $G \not\in \{K_3, K_5\}$. If $\Delta(G) \leq 5$ and $G$ does not contain configurations $C_1, \dots, C_5$, then the graph $G_E$ is a forest.
\end{lemma}

\begin{proof}
Let $H$ = $G_E$ and suppose for a contradiction that $H$ contains a cycle $C$. Suppose further that there is $v \in V(C)$ with $d(v)=2$, and let $N(v) = \{u,w\}$. Since $C$ is a cycle in $H$, we have that $d(u), d(w) \in \{2,4\}$. Furthermore, since $G$ does not contain configuration $C_1$, we have that $uw \in E(G)$. Now $G \neq K_3$, so at least one of $u$ and $w$ has degree $4$. It follows that $u,v$ and $w$ form configuration $C_5$, a contradiction. Thus, if $C$ is a cycle in $H$, then $d_G(v) = 4$ for all vertices $v \in V(C)$. Since $G$ does not contain configuration $C_5$, it immediately follows that $|C| > 3$.

Let $uv$ be an edge of $C$. Let $t_1, t_2, t_3$ be the neighbours of $u$ apart from $v$ and let $w_1,w_2,w_3$ be the neighbours of $v$ apart from $u$. Note that, since $uv$ is an edge of $C$, at least one of $t_1, t_2, t_3$ has degree $4$. Similarly, at least one of $w_1,w_2,w_3$ has degree $4$. Now, $u$ and $v$ do not have $3$ common neighbours, since otherwise $G$ contains configuration $C_5$, a contradiction. Furthermore, since $G$ does not contain configuration $C_3$, the vertices $u$ and $v$ have at most one common neighbour. Thus, in what follows, we allow the possibility that $t_1 = w_1$, but always assume that $t_2, t_3 \not\in \{w_1,w_2,w_3\}$ and $w_2, w_3 \not\in \{t_1,t_2,t_3\}$.

Suppose first that $t_1t_2 \not\in E(G)$. Since $G$ does not contain configuration~$C_4$, we must have that $w_1w_2, w_2w_3, w_1w_3 \in E(G)$. Otherwise, since $t_3 \not\in \{w_1,w_2,w_3\}$, we have that $G$ contains configuration $C_4$, a contradiction. But now the vertices $w_1, w_2, w_3$ form a clique, and at least one of them has degree $4$. It follows that $G$ contains configuration $C_3$, a contradiction. 

It follows that all of the edges $t_1t_2, t_1t_3, w_1w_2, w_1w_3 \in E(G)$. As a consequence, $t_1 \neq w_1$, otherwise this vertex would have degree $6$, which is larger than $\Delta(G)$. Thus $\{t_1,t_2,t_3\} \cap \{w_1,w_2,w_3\} = \emptyset$. With this extra information, the argument above shows that, in fact, if any edge amongst $t_1,t_2,t_3$ is not in $E(G)$, then $w_1,w_2,w_3$ induce a clique. Thus, either $\{t_1,t_2,t_3\}$ or $\{w_1,w_2,w_3\}$ induce a clique, which again gives a contradiction since $G$ does not contain configuration $C_3$.

\end{proof}

The proof of Theorem~\ref{th:main} now follows easily.

\begin{proof}[Proof of Theorem~\ref{th:main}]
Let $G$ be a smallest counterexample to the theorem. By Lemma~\ref{lem:ReducibleConfigs}, the graph $G$ does not contain configurations $C_1, \dots, C_5$. Thus, by Lemma~\ref{lem:Structure}, the graph $G_E$ is a forest. But now $G$ admits a good path decomposition by Theorem~\ref{thm:Pyber}, a contradiction.
\end{proof}

\section{Acknowledgements}

The authors wish to thank Fran\c{c}ois Dross for useful discussions. The second author was supported by ERC Advanced Grant GRACOL, project number 320812.

\end{document}